\documentclass[11pt]{amsart}

\usepackage[TS1,T1]{fontenc}
\usepackage[latin1]{inputenc}
\usepackage{amsfonts,amssymb,amsmath,amsthm,enumerate,latexsym}

\def\R{{\mathbb {R}}}
\def\N{{\mathbb {N}}}
\newcommand{\ve}{\varepsilon}
\def\F{{\mathcal {F}}}
\def\A{{\mathcal {A}}}
\def\eps{{\varepsilon}}
\newcommand{\J}{\mathcal{J}}
\newcommand{\p}{\partial}
\newcommand{\PP}{\mathcal{P}}
\newcommand{\HH}{\mathcal{H}}

\newtheorem{teo}{Theorem}[section]
\newtheorem{lema}[teo]{Lemma}
\newtheorem{prop}[teo]{Proposition}
\newtheorem{corol}[teo]{Corollary}

\theoremstyle{remark}

\theoremstyle{definition}
\newtheorem{defi}[teo]{Definition}

\numberwithin{equation}{section}

\begin{document}

\title[Existence of solution]{Existence of solution to a critical trace equation with variable exponent}
\author[J. Fern\'andez Bonder, N. Saintier and A. Silva]{Juli\'an Fern\'andez Bonder, Nicolas Saintier and Anal\'ia Silva}

\address[J. Fern\'andez Bonder and A. Silva]{IMAS - CONICET and Departamento de Matem\'atica, FCEyN - Universidad de Buenos Aires, Ciudad Universitaria, Pabell\'on I  (1428) Buenos Aires, Argentina.}

\address[N. Saintier]{Instituto de Ciencias, Universidad Nacional de General Sarmiento, Juan María Gutierrez 1150 Los Polvorines - Pcia de Bs. As. - Argentina and Departamento de Matem\'atica, FCEyN - Universidad de Buenos Aires, Ciudad Universitaria, Pabell\'on I  (1428) Buenos Aires, Argentina.}

\email[J. Fernandez Bonder]{jfbonder@dm.uba.ar}

\urladdr[J. Fernandez Bonder]{http://mate.dm.uba.ar/~jfbonder}

\email[A. Silva]{asilva@dm.uba.ar}

\email[N. Saintier]{nsaintie@dm.uba.ar, nsaintie@ungs.edu.ar}

\urladdr[N. Saintier]{http://mate.dm.uba.ar/~nsaintie}




\begin{abstract}
In this paper we study sufficient local conditions for the existence of non-trivial solution to a critical equation for the $p(x)-$Laplacian where the critical term is placed as a source through the boundary of the domain. The proof relies on a suitable generalization of the concentration--compactness principle for the trace embedding for variable exponent Sobolev spaces and the classical mountain pass theorem.
\end{abstract}

\maketitle
\section{introduction}
Let $\Omega\subset\R^N$ be a smooth bounded open set. The purpose of this article is the study of the existence of a nontrivial solution to the critical trace equation
\begin{equation}\label{1.1}
\begin{cases}
 -\Delta_{p(x)}u + h|u|^{p(x)-2}u  =0 &\quad \text{in } \Omega, \\
 |\nabla u|^{p(x)-2}\partial_\nu u=  |u|^{r(x)-2}u &\quad \text{on } \partial \Omega,
\end{cases}
\end{equation}
where $\Delta_{p(x)}u=-\text{div}(|\nabla u|^{p(x)-2}\nabla u)$ is the $p(x)$-Laplacian corresponding to some given function $p\colon\bar\Omega\to (1,+\infty)$  (notice that when $p$ is constant we recover the usual $p$-Laplacian),  
 $\p_\nu$ is the outer normal derivative, and $h$ is a smooth function satisfying some coercivity assumption (see the definition of the norm in (\ref{coercive}) below). The exponents $p\colon\bar\Omega\to (1,+\infty)$ and $r\colon \p\Omega\to [1,+\infty)$ are continuous functions that verify
\begin{equation}\label{pr}
1<p^-:=\inf_{x\in \Omega}p(x)\le p^+:=\sup_{x\in \Omega}p(x)<N\quad \text{and}\quad r(x)\le p_*(x) = \frac{(N-1)p(x)}{N-p(x)}.
\end{equation}
The exponent $p_*$ is critical from the point of view of the Sobolev trace emebdding   $W^{1,p(x)}(\Omega)\hookrightarrow L^{r(x)}(\p\Omega)$ (see Theorems \ref{trace} and \ref{Compact_teo} in section 2 below for a precise statement). 

We focus in this paper on the critical problem for \eqref{1.1} in the sense that we will assume from now on that 
\begin{equation}\label{r.critic}
\A_T:=\{x\in\p\Omega\colon r(x)=p_*(x)\}\neq\emptyset.
\end{equation}
Under this assumption the embedding $W^{1,p(x)}(\Omega)\hookrightarrow L^{r(x)}(\p\Omega)$ is generally not compact so that the existence of a non-trivial solution to \eqref{1.1} is a non-trivial problem. Our main purpose is to find conditions on $p$, $r$ and $\Omega$ in the spirit of \cite{Adi}, \cite{Escobar}, and \cite{FBS}, where this kind of problem has been considered in the constant exponent case, ensuring the  existence of a non-trivial solution to \eqref{1.1}. 

\medskip

Observe that problem \eqref{1.1} is variational in the sense that weak solutions are critical points of the associated functional
\begin{equation}\label{F}
\F(u):= \int_{\Omega} \frac{1}{p(x)}\left[|\nabla u|^{p(x)} + h |u|^{p(x)}\right]\, dx - \int_{\p\Omega} \frac{1}{r(x)}|u|^{r(x)}\, dS,
\end{equation}
where $dS$ denotes the boundary measure. This functional $\F$ is well defined in $W^{1,p(x)}(\Omega)$ thanks to \eqref{pr} (see  Theorems \ref{trace} in section 2 below). 
The main tool available in order to find critical points for $C^1$ functionals in Banach spaces is the well known Mountain Pass Theorem (MPT). The MPT has two types of hypotheses, geometrical and topological.

For the functional $\F$ it is fairly easy to see that when $p^+<r^-$ the geometrical hypotheses of the MPT are satisfied. The topological hypothese is the so-called Palais--Smale condition that requires for a sequence of approximate critical points to be precompact. When $r(x)$ is {\em uniformly subcritical}, i.e.
\begin{equation}\label{unif_subcrit}
 \inf_{x\in \p\Omega}(p_*(x)-r(x))>0, 
\end{equation} 
the immersion $W^{1,p(x)}(\Omega)\hookrightarrow L^{r(x)}(\p\Omega)$ is compact. It is then  straightforward to check that the Palais--Smale condition is satisfied for every energy level $c$.

Notice that there are some cases where the subcriticality is violated but still the immersion is compact.
In fact, in \cite{FBSS3} the authors find conditions on the exponents $p$ and $r$ such that $\A_T\neq\emptyset$ but the immersion remains compact. This type of conditions were first discovered in \cite{MOSS} where the embedding $W_0^{1,p(x)}(\Omega)\hookrightarrow L^{q(x)}(\Omega)$, $q(x)\le p^*(x):=Np(x)/(N-p(x))$ was analyzed. The result in \cite{FBSS3} shows that if the criticality set $\A_T$ is ``small'' and we have a control on how the exponent $r$ reaches $p_*$ at the criticality set, then the immersion $W^{1,p(x)}(\Omega)\hookrightarrow L^{r(x)}(\p\Omega)$ remains compact, and so the existence of solutions to \eqref{1.1} follows as in the subcritical case.

However, in the general case $\A_T\not=\emptyset$, the present paper is, up to our knowledge, the first work regarding the existence of solutions for \eqref{1.1}.

Recently, in \cite{FBSS3}, the authors analyzed the problem of the existence of extremals for the immersion $W^{1,p(x)}(\Omega)\hookrightarrow L^{r(x)}(\p\Omega)$, that is functions realizing the infimum in 
$$ 0<T(p(\cdot), r(\cdot), \Omega) 
 := \inf_{v\in W^{1,p(x)}(\Omega)}  \frac{\|v\|_{W^{1,p(x)}(\Omega)}}{\|v\|_{L^{r(x)}(\p\Omega)}}.
$$
In \cite{FBSS3} the main tool used to deal with the existence of extremals problem is the extension of the celebrated Concentration--Compactness Principle (CCP) of P.L. Lions to the variable exponent case. In the case of the immersion $W^{1,p(x)}_0(\Omega)\hookrightarrow L^{q(x)}(\Omega)$ this was done independently by \cite{FBS1} and \cite{Fu} (see also \cite{FBSS1} where a refinement of the result was obtained). For the trace immersion, this result was proved in the above mentioned paper \cite{FBSS3}.

\medskip

In order to state our main results we need to introduce some notation.
Given some nonempty, closed subset $\Gamma\subset \partial\Omega$ (possibly empty), we consider the space $W^{1,p(x)}_\Gamma(\Omega)$ defined by
$$ W^{1,p(x)}_\Gamma(\Omega):=
 \overline{\{u \in C^1(\bar\Omega)\colon u=0\text{ in a neighborhood of }\Gamma\}}, $$
the closure being taken in the $\|\cdot\|_{1,p(x)}-$norm. This is the space of functions vanishing on $\Gamma$. Observe that $W^{1,p(x)}_\emptyset(\Omega) = W^{1,p(x)}(\Omega)$ and, more generally, that  $W^{1,p(x)}_\Gamma(\Omega) = W^{1,p(x)}(\Omega)$ if and only if $\Gamma$ has $p(x)-$capacity zero. See \cite{Harjuleto}.
Given a critical point $x\in\A_T$, we define the localized best Sobolev trace constant $\bar T_x$ around $x$ by
\begin{equation}\label{LocBestCste}
 \bar T_x = \sup_{\eps>0} T(p(\cdot), r(\cdot), \Omega_\eps, \Gamma_\eps), 
\end{equation}
where
\begin{equation}\label{LocBestCste2}
T(p(\cdot), r(\cdot), \Omega_\eps, \Gamma_\eps)  = \inf_{v\in W^{1,p(x)}_{\Gamma_\eps}(\Omega_\eps)}\frac{\|v\|_{W^{1,p(x)}(\Omega_\eps)}}{\|v\|_{L^{r(x)}(\p\Omega_\eps)}} \quad \text{and}\quad \Omega_\eps = \Omega\cap B_\eps(x),\ \Gamma_\eps = \Omega\cap \partial B_\eps(x).
\end{equation}

Our first result states that the functional $\F$ defined in (\ref{F}) verifies the Palais--Smale condition for any energy level $c$ below a {\em critical energy level} $c^*$ given by 
$$
c^* := \inf_{x\in\A_T} \left(\frac{1}{p(x)} - \frac{1}{p_*(x)}\right) \bar T_{x}^{\frac{p(x)p_*(x)}{p_*(x)-p(x)}}. 
$$
As an immediate corollary of this result, we obtain appplying the MPT the existence of a solution to \eqref{1.1} provided there exists a function $v\in W^{1,p(x)}(\Omega)$ such that
\begin{equation}\label{Fc*}
\sup_{t>0} \F(tv) < c^*.
\end{equation}
The rest of the paper is devoted to find conditions on $p, r$ and $\Omega$ that allow us to construct a function $v$ that satisfies \eqref{Fc*}. The idea used in the construction of such $v$ is to rescale and truncate an extremal for the Sobolev trace immersion
$$ \bar K(N,p)^{-1} =  \inf_{f\in C^{\infty}_c(\R^N)}
 \frac{\displaystyle \left(\int_{\R^N_+} |\nabla f|^{p}\,dx\right)^{\frac{1}{p}}}
      {\displaystyle \left(\int_{\R^{N-1}} |f|^{p_*}\, dy\right)^{\frac{1}{p_*}}}. 
$$
These extremals were found by Nazaret in \cite{Nazaret} by means of mass transportation methods extending the well known result of Escobar in \cite{Escobar} where the case $p=2$ was studied. These extremals are of the form
\begin{equation}\label{Extremal}
V_{\lambda,y_0}(y,t) = \lambda^{-\frac{N-p}{p-1}}V\left(\tfrac{y-y_0}{\lambda}, \tfrac{t}{\lambda}\right), \qquad y\in \R^{N-1},\ t>0, 
\end{equation} 
where
\begin{equation}\label{ExtremalProfile}
V(y,t) = r^{-\frac{N-p}{p-1}},\qquad r=\sqrt{(1+t)^2 + |y|^2}.
\end{equation}

Similar ideas were used recently in \cite{FBSS3} were the existence problem for extremals in the critical Sobolev trace immersion was studied. These ideas were also previously used for \eqref{1.1} in the constant exponent case by Adimurthi-Yadava \cite{Adi}, Escobar \cite{Escobar}, and
Fernandez Bonder and Saintier in \cite{FBS}. Let us mentioned that these ideas are classical when dealing with critical equations. They go back to the seminal paper of Aubin \cite{Aubin} and 
Brezis and Nirenberg \cite{BN} and have been widely used since then in the constant exponent case (see e.g. \cite{Aubin3,DHL,D,DH,Escobar,ER,Faget,FBS, HV, Saintier2, Saintier3,Schoen} and references therein). In the variable setting we refer to the recent paper  \cite{FBSS2} where analogous results for the critical problem with Dirichlet boundary conditions have been obtained.

\subsection*{Organization of the paper} The rest of the paper is organized as follows. In Section 2, we collect some preliminaries on variable exponent spaces that will be used throughout the paper. In Section 3 we give an existence criteria for solutions, namely condition \eqref{Fc*}. In section 4 we  give  conditions that ensure the validity of such criteria. We leave for the Appendix some asymptotic expansions needed in the proof of our results.

\section{Preliminaries on variable exponent Sobolev spaces}

In this section we review some preliminary results regarding Lebesgue and Sobolev spaces with variable exponent. All of these results and a comprehensive study of these spaces can be found in \cite{libro}.

\medskip

We denote by $\PP(\Omega)$ the set of Lebesgue measurable functions $p\colon \Omega\to [1,\infty)$.
Given $p\in\PP(\Omega)$ we consider the variable exponent Lebesgue space $L^{p(x)}(\Omega)$ defined by 
$$ L^{p(x)}(\Omega) = \Big\{u\in L^1_{\text{loc}}(\Omega) \colon \int_\Omega|u(x)|^{p(x)}\,dx<\infty\Big\}. $$
This space is endowed with the (Luxembourg) norm
$$ \|u\|_{L^{p(x)}(\Omega)} = \|u\|_{p(x)} :=\inf\Big\{\lambda>0:\int_\Omega\Big|\frac{u(x)}{\lambda}\Big|^{p(x)}\,dx\leq 1\Big\}. $$

The following H\"older-type inequality is proved in \cite{Fan, KR} (see also \cite{libro}, pp. 79, Lemma 3.2.20 (3.2.23)):

\begin{prop}[H\"older-type inequality]\label{Holder}
Let $f\in L^{p(x)}(\Omega)$ and $g\in L^{q(x)}(\Omega)$. Then the following inequality holds
$$ \|fg\|_{L^{s(x)}(\Omega)}\le  \Big( \Big(\frac{s}{p}\Big)^+ + \Big(\frac{s}{q}\Big)^+\Big) \|f\|_{L^{p(x)}(\Omega)}\|g\|_{L^{q(x)}(\Omega)}, $$
where
$$ \frac{1}{s(x)} = \frac{1}{p(x)} + \frac{1}{q(x)}. $$
\end{prop}

The following proposition, also proved in \cite{KR}, will be most useful (see also \cite{libro}, Chapter 2, Section 1): 

\begin{prop}\label{norma.y.rho}
Set $\rho(u):=\int_\Omega|u(x)|^{p(x)}\,dx$. For $u\in L^{p(x)}(\Omega)$ and $\{u_k\}_{k\in\N}\subset L^{p(x)}(\Omega)$, we have
\begin{align}
& u\neq 0 \Rightarrow \Big(\|u\|_{L^{p(x)}(\Omega)} = \lambda \Leftrightarrow \rho(\frac{u}{\lambda})=1\Big).\\
& \|u\|_{L^{p(x)}(\Omega)}<1 (=1; >1) \Leftrightarrow \rho(u)<1(=1;>1).\\
& \|u\|_{L^{p(x)}(\Omega)}>1 \Rightarrow \|u\|^{p^-}_{L^{p(x)}(\Omega)} \leq \rho(u) \leq \|u\|^{p^+}_{L^{p(x)}(\Omega)}.\\
& \|u\|_{L^{p(x)}(\Omega)}<1 \Rightarrow \|u\|^{p^+}_{L^{p(x)}(\Omega)} \leq \rho(u) \leq \|u\|^{p^-}_{L^{p(x)}(\Omega)}.\\
& \lim_{k\to\infty}\|u_k\|_{L^{p(x)}(\Omega)} = 0 \Leftrightarrow \lim_{k\to\infty}\rho(u_k)=0.\\
& \lim_{k\to\infty}\|u_k\|_{L^{p(x)}(\Omega)} = \infty \Leftrightarrow \lim_{k\to\infty}\rho(u_k) = \infty.
\end{align}
\end{prop}

The following Lemma is the extension to variable exponents of the well-known Brezis--Lieb
Lemma (see \cite{Brezis-Lieb}). The proof is analogous to that of \cite{Brezis-Lieb}. See Lemma 3.4 in \cite{FBS1}

\begin{lema}\label{Brezis-Lieb}
Let $f_n\to f$ a.e and $f_n\rightharpoonup f$ in $L^{p(x)}(\Omega)$ then
$$ \lim_{n\to\infty} \left(\int_\Omega|f_n|^{p(x)}\, dx - \int_\Omega |f-f_n|^{p(x)}\, dx\right) = \int_\Omega |f|^{p(x)}\, dx. $$
\end{lema}

We now define the variable exponent Lebesgue spaces on $\partial \Omega$. 
First we denote by $\PP(\partial\Omega)$  the set of $\HH^{N-1}-$measurable functions $r\colon\partial\Omega\to [1,\infty)$. 
We then assume that $\Omega$ is $C^1$ so that $\partial \Omega$ is a $(N-1)-$dimensional $C^1$ immersed manifold on $\R^N$ (although the trace theorem require less regularity on $\partial \Omega$, the $C^1$ regularity will be enough for our purposes). Therefore the boundary measure agrees with the $(N-1)-$Hausdorff measure restricted to $\partial\Omega$. We denote this measure by $dS$. Then, the Lebesgue spaces on $\partial\Omega$ are defined as
$$ L^{r(x)}(\partial\Omega):= \Big\{ u\in L^1_{\text{loc}}(\partial\Omega, dS)\colon \int_{\partial\Omega} |u(x)|^{r(x)}\, dS<\infty\Big\}, $$
and the corresponding (Luxemburg) norm is given by
$$ \|u\|_{L^{r(x)}(\partial\Omega)} = \|u\|_{r(x), \partial\Omega} := \inf\Big\{\lambda>0\colon \int_{\partial\Omega} \Big|\frac{u(x)}{\lambda}\Big|^{r(x)}\, dS\le 1\Big\}. $$

\medskip

We can define in a similar way the variable exponent Sobolev space $W^{1,p(x)}(\Omega)$ by
$$ W^{1,p(x)}(\Omega) = \{u\in L^{p(x)}(\Omega) \colon \partial_i u\in L^{p(x)}(\Omega) \text{ for } i=1,\dots,N\}, $$
where $\partial_i u = \frac{\partial u}{\partial x_i}$ is the $i^{th}-$distributional partial derivative of $u$.
This space has a corresponding modular given by \index{$\rho_{1,p(x)}$}
$$ \rho_{1,p(x)}(u) := \int_\Omega |u|^{p(x)} + |\nabla u|^{p(x)}\, dx, $$
and so the corresponding norm for this space is 
\begin{equation}\label{NormW1p}
 \|u\|_{W^{1,p(x)}(\Omega)} = \| u\|_{1,p(x)} : = \inf\Big\{\lambda>0\colon \rho_{1, p(x)}\left(\frac{u}{\lambda}\right)\leq 1\Big\}. 
\end{equation}
The $W^{1,p(x)}(\Omega)$ norm can also be defined as  $\|u\|_{p(x)} + \| \nabla u \|_{p(x)}$. Both norms turn out to be equivalent but we use the first one for convenience.

The following Sobolev trace Theorems are proved in \cite{Fan}.

\begin{teo}\label{trace}
Let $\Omega\subseteq\R^N$ be an open bounded domain with Lipschitz boundary and let $p\in \PP(\Omega)$ be such that $p\in W^{1,\gamma}(\Omega)$ with $1\leq p_{-}\leq p^{+}<N<\gamma$. Then there is a continuous boundary trace embedding $W^{1,p(x)}(\Omega)\subset L^{p_*(x)}(\partial\Omega)$.
\end{teo}

\noindent We used the following notation: for a $\mu-$measurable function $f$ we denote $f^+ := \sup f$ and $f^- := \inf f$, where by $\sup$ and $\inf$ we denote the essential supremum and essential infimum respectively with respect to the measure $\mu$.

The regularity assumption on $p$ can be relaxed when the exponent $r$ is unifortmly subcritical in the sense of (\ref{unif_subcrit}). It holds 

\begin{teo}\label{Compact_teo}
Let $\Omega\subset\R^N$ be an open bounded domain with Lipschitz boundary. Suppose that $p\in C^0(\bar{\Omega})$ and $1<p^-\leq p^+<N$. If $r\in \PP(\partial\Omega)$ is uniformly subcritical then the boundary trace embedding $W^{1,p(x)}(\Omega)\to L^{r(x)}(\partial\Omega)$ is compact.
\end{teo}

\begin{corol}
Let $\Omega\subset\R^N$ be an open bounded domain with Lipschitz boundary. Suppose that $p\in C^0(\bar{\Omega})$ and $1<p_{-}\leq p_{+}<N$. If $r\in C^0(\partial\Omega)$ satifies the condition
$$
1\leq r(x)<p_*(x)\quad x\in\partial\Omega
$$
then there is a compact boundary trace embedding $W^{1,p(x)}(\Omega)\to L^{r(x)}(\partial\Omega)$
\end{corol}

For much more on these spaces, we refer to \cite{libro}.

\section{Existence criteria for solutions}

We consider the equation
\begin{equation}\label{MainEq traza}
\begin{cases}
 -\Delta_{p(x)}u + h(x)|u|^{p(x)-2}u  =0 &\quad \text{in } \Omega, \\
 |\nabla u|^{p(x)-2}\partial_\nu u =  |u|^{r(x)-2}u &\quad \text{on } \partial \Omega,
\end{cases}
\end{equation}
where $\Omega\subset \R^N$ is a bounded domain, $p\in \PP(\Omega)$, $1<p^-\le p^+<N$, and $r\in \PP(\p\Omega)$ is critical in the sense that $\A_T\neq\emptyset$ where $\A_T$ is defined in (\ref{r.critic}). 
In order to study \eqref{MainEq traza} by means of variational methods, we need to consider the functional $\F\colon W^{1,p(x)}(\Omega)\to\R$ defined by
\begin{equation}\label{DefJt}
\F(u):= \int_\Omega \frac{1}{p(x)}\Big(|\nabla u|^{p(x)} + h(x)|u|^{p(x)}\Big)\, dx - \int_{\partial\Omega} \frac{1}{r(x)} |u|^{r(x)}\, dS.
\end{equation}
Then $u\in W^{1,p(x)}(\Omega)$ is a weak solution of \eqref{MainEq traza} if and only if $u$ is a critical point of $\F$.
We need to assume that the smooth function $h$ is such that the functional
\begin{equation}\label{DefI}
\J(u):=\int_\Omega |\nabla u|^{p(x)}+h(x)|u|^{p(x)} \,dx
\end{equation}
is coercive in the sense that the norm
\begin{equation}\label{coercive}
\|u\|:= \inf \left\{\lambda>0\, \int_\Omega \left|\frac{\nabla u + h(x)u(x)}{\lambda}\right|^{p(x)}\,dx \le 1\right\}, 
\end{equation}
is equivalent to the usual norm  $\|\cdot\|_{1,p(x)}$ of  $W^{1,p(x)}(\Omega)$ defined in (\ref{NormW1p}).

\medskip

It is not difficult to prove that $\F$ verifies the geometrical assumptions of the Mountain Pass Theorem (cf. the proof of Theorem \ref{teoMPt}). 
The first non-trivial result needed to apply the Mountain Pass Theorem is to check that the Palais--Smale condition holds below some critical energy level $c^*$ that can be computed explicitly in terms of the Sobolev trace constant $T(p(\cdot), r(\cdot), \Omega)$. Once this fact is proved, the main difficulty is to exhibit some Palais--Smale sequence with energy below the critical level $c^*$.

This approach has been used with success by several authors for treating critical elliptic problems, starting with the seminal papers of \cite{Aubin, Aubin2,BN}. See, for instance \cite{Aubin3,DHL,D,DH,Escobar,ER,Faget,FBS, HV, Saintier2, Saintier3,Schoen} and references therein. 

\medskip

Our first result gives an explicit value of the energy below which the functional $\F$ satisfy the Palais--Smale condition.

\begin{teo}\label{teoPScondt}
Assume that $h$ is such that $\J$ is coercive (see \eqref{coercive} above).  Then the functional $\F$ satisfies the Palais--Smale condition at level
$$
0<c<\inf_{x\in\A_T} \left(\frac{1}{p(x)} - \frac{1}{p_*(x)}\right) \bar T_{x}^{\frac{p(x)p_*(x)}{p_*(x)-p(x)}}.
$$
\end{teo}

\begin{proof}
Let $\{u_k\}_{k\in\N}\subset W^{1,p(x)}(\Omega)$ be a Palais--Smale sequence for $\F$. Recall that this means that the sequence $\{\F(u_k)\}_{k\in\N}$ is bounded, and that $\F'(u_k)\to 0$ strongly in the dual space $(W^{1,p(x)}(\Omega))'$.
Recalling that the functional $\J$ defined by (\ref{DefI}) is assumed to be coercive (see the norm (\ref{coercive}) above), it then follows that $\{u_k\}_{k\in\N}$ is bounded in $W^{1,p(x)}(\Omega)$. In fact, for $k$ large, we have that
\begin{align*}
c+1+o(1)\|u_k\| &\ge \F(u_k) - \frac{1}{r^-}\langle \F'(u_k), u_k\rangle \\
&\ge \big(\frac{1}{p^+}-\frac{1}{r^-}\big) \int_\Omega |\nabla u_k|^{p(x)} + h(x) |u_k|^{p(x)}\, dx 
+ \int_{\partial\Omega} \big(\frac{1}{r^-}-\frac{1}{r(x)}\big) |u_k|^{r(x)}\, dS\\
&\ge \big(\frac{1}{p^+}-\frac{1}{r^-}\big) \int_\Omega |\nabla u_k|^{p(x)} + h(x) |u_k|^{p(x)}\, dx 
= \big(\frac{1}{p^+}-\frac{1}{r^-}\big) \J(u_k).
\end{align*}

\medskip

We may thus assume that $u_k\rightharpoonup u$ weakly in $W^{1,p(x)}(\Omega)$. We claim that $u$ turns out to be a weak solution to \eqref{MainEq traza}. The proof of this fact follows closely the one in \cite{Saintier} and this argument is taken from \cite{Hebey, Evans2}, where the constant exponent case is treated.

In fact, since $\{u_k\}_{k\in\N}$ is a Palais--Smale sequence, we have that
$$
\langle\F'(u_k), v\rangle = \int_\Omega |\nabla u_k|^{p(x)-2}\nabla u_k \nabla v\, dx + \int_\Omega h |u_k|^{p(x)-2}u_k v\, dx - \int_{\p\Omega} |u_k|^{r(x)-2}u_k v\, dS = o(1)
$$
for any $v\in C^1(\bar\Omega)$. Without loss of generality, we can assume that $u_k\to u$ a.e. in $\Omega$, $\HH^{N-1}-$a.e. in $\p\Omega$,  and in $L^{p(x)}(\Omega)$. It is easy to see, from standard integration theory, that
$$
 \int_\Omega h |u_k|^{p(x)-2}u_k v\, dx \to  \int_\Omega h |u|^{p(x)-2}u v\, dx \quad \text{and}\quad  \int_{\p\Omega} |u_k|^{r(x)-2}u_k v\, dS \to  \int_{\p\Omega} |u|^{r(x)-2}u v\, dS,
$$
so the claim will follows if we show that
$$
\int_\Omega |\nabla u_k|^{p(x)-2}\nabla u_k \nabla v\, dx\to \int_\Omega |\nabla u|^{p(x)-2}\nabla u \nabla v\, dx.
$$
This is a consequence of the monotonicity of the $p(x)$-Laplacian. We can assume that there exist $\xi\in (L^{p'(x)}(\Omega))^N$ such that
$$
|\nabla u_k|^{p(x)-2}\nabla u_k \rightharpoonup \xi \quad\mbox{weakly in }(L^{p'(x)}(\Omega))^N.
$$
The idea is to show that $\nabla u_k\to \nabla u$ a.e. in $\Omega$, then this will imply that $\xi = |\nabla u|^{p(x)-2}\nabla u$ and thus, the claim.

Let $\delta>0$ then, by Egoroff's Theorem, there exists $E_\delta\subset \Omega$ such that $|\Omega\setminus E_\delta|<\delta$ and $u_k\to u$ uniformly in $E_\delta$. As a consequence, given $\eps>0$, there exists $k_0\in \N$ such that $|u_k(x) - u(x)|<\eps/2$ for $x\in E_\delta$ and for any $k\ge k_0$.

Define the truncation $\beta_\eps$ as
$$
\beta_\eps(t) = \begin{cases}
-\eps & \text{if } t\le -\eps\\
t & \text{if } -\eps< t <\eps\\
\eps& \text{if } t\ge \eps.
\end{cases}
$$

Now we make use of the following well known monotonicity inequality
\begin{equation}\label{monoton.ineq}
(|x|^{p-2}x - |y|^{p-2}y)(x-y)\ge 0
\end{equation}
which is valid for any $x,y\in \R^N$ and $p\ge 1$ and we obtain
$$
(|\nabla u_k|^{p(x)-2}\nabla u_k - |\nabla u|^{p(x)-2}\nabla u)\nabla \left(\beta_\eps(u_k-u)\right)
\ge 0,
$$
since $\nabla \beta_\eps(u_k-u) = \nabla u_k - \nabla u$ in $E_\delta$ and $\nabla \beta_\eps(u_k-u) = 0$ in $\Omega\setminus E_\delta$. Therefore, we obtain
\begin{align*}
\int_{E_\delta}(|\nabla u_k|^{p(x)-2}\nabla u_k - |\nabla u|^{p(x)-2}&\nabla u)(\nabla u_k- \nabla u)\, dx \\
&\le \int_\Omega (|\nabla u_k|^{p(x)-2}\nabla u_k - |\nabla u|^{p(x)-2}\nabla u)\nabla \beta_\eps(u_k-u)\, dx.
\end{align*}
Now, observe that $\beta_\eps(u_k-u)\rightharpoonup 0$ weakly in $W^{1, p(x)}_0(\Omega)$ and so
$$
\int_\Omega |\nabla u|^{p(x)-2}\nabla u\nabla \beta_\eps(u_k-u)\, dx\to 0.
$$
Now, for $k$ sufficiently large, we obtain that
$$
\int_\Omega |\nabla u_k|^{p(x)-2}\nabla u_k\nabla \beta_\eps(u_k-u)\, dx\le C\eps
$$
for some constant $C>0$. In fact, since $\beta_\eps(u_k-u)$ is bounded in $W^{1,p(x)}(\Omega)$,
$$
\langle \F'(u_k), \beta_\eps(u_k-u)\rangle = o(1),
$$
so that
$$
\int_\Omega |\nabla u_k|^{p(x)-2}\nabla u_k\nabla \beta_\eps(u_k-u)\, dx = o(1) + I_1 + I_2,
$$
where
$$
|I_1| = \Big|\int_{\p\Omega} |u_k|^{r(x)-2}u_k \beta_\eps(u_k-u)\, dS \Big| \le \eps \int_{\p\Omega} |u_k|^{r(x)-1}\, dS \le C\eps
$$
and
$$
|I_2| = \Big|\int_\Omega h |u_k|^{p(x)-2}u_k \beta_\eps(u_k-u)\, dx\Big|\le \eps \|h\|_{\infty}\int_\Omega |u_k|^{p(x)-1}\, dx\le C\eps.
$$
As a consequence, we get that
$$
0\le \limsup_{k\to\infty} \int_{E_\delta}(|\nabla u_k|^{p(x)-2}\nabla u_k - |\nabla u|^{p(x)-2}\nabla u)(\nabla u_k- \nabla u)\, dx \le C\eps.
$$
Since $\eps>0$ is arbitrary, it follows that $(|\nabla u_k|^{p(x)-2}\nabla u_k - |\nabla u|^{p(x)-2}\nabla u)(\nabla u_k- \nabla u)\to 0$ strongly in $L^1(E_\delta)$ and thus, up to a subsequence, also a.e. in $E_\delta$. By a standard diagonal argument, we can assume that $(|\nabla u_k|^{p(x)-2}\nabla u_k - |\nabla u|^{p(x)-2}\nabla u)(\nabla u_k- \nabla u)\to 0$ a.e. in $E_\delta$ for every $\delta>0$ and so the convergence holds a.e. in $\Omega$.

Finally, it is easy to see that $(|x_k|^{p-2}x_k - |x|^{p-2}x)(x_k-x)\to 0$ for $x_k,x\in \R^N$ and $p\ge 1$ imply that $x_k\to x$, so we get that $\nabla u_k\to \nabla u$ a.e. in $\Omega$. This concludes the proof of the claim.

\medskip

By the Concentration Compactness Principle for variable exponents in the trace case, see \cite{FBSS3}, it holds that
\begin{align}
& |u_k|^{r(x)}\, dS \rightharpoonup \nu = |u|^{r(x)}\, dS + \sum_{i\in I} \nu_i \delta_{x_i} \quad \mbox{weakly in the sense of measures,}\label{CCPtrace}\\
& |\nabla u_k|^{p(x)}\, dx \rightharpoonup \mu \ge |\nabla u|^{p(x)} \, dx + \sum_{i\in I} \mu_i \delta_{x_i} \quad \mbox{weakly in the sense of measures,}\label{CCPtrace2}\\
&\bar T_{x_i} \nu_i^{1/p^*(x_i)}\le \mu_i^{1/p(x_i)}\label{CCPtrace3},
\end{align}
where $I$ is a countable set, $\{\nu_i\}_{i\in I}$ and $\{\mu_i\}_{i\in I}$ are positive numbers, the points $\{x_i\}_{i\in I}$ belong to the critical set $\A_T\subset\partial\Omega$, and $\bar T_{x_i}$ is the localized best Sobolev constant around $x_i$ defined by (\ref{LocBestCste2}). 

It is not difficult to check that $v_k:=u_k-u$ is a PS--sequence for the functional $\tilde \F$ defined by 
$$\tilde \F(v):=\F(v)-\int_\Omega \frac{1}{p(x)} h|v|^{p(x)}\,dx.$$ 
Now, by the Brezis-Lieb lemma \ref{Brezis-Lieb} we get
\begin{align*}
\F(u_k)-\F(u) & =   \int_\Omega \frac{1}{p(x)}\Big[|\nabla v_k|^{p(x)} + h |v_k|^{p(x)}\Big]\, dx - \int_{\partial\Omega}\frac{1}{r(x)} |v_k|^{r(x)}\, dS + o(1)\\
& =  \tilde \F(v_k) + \int_\Omega \frac{1}{p(x)}  h |v_k|^{p(x)}\, dx + o(1)\\
& =  \tilde \F(v_k) + o(1).
\end{align*}
Independently since $u$ is a weak solution of \eqref{MainEq traza},  and recalling that $p^+<r^-$, we have 
\begin{equation*}
\begin{split}
 \F(u) & \ge \frac{1}{p^+} \int_\Omega \left(|\nabla u|^{p(x)}+h(x)|u|^{p(x)}\right)\,dx
           - \frac{1}{r^-} \int_{\partial\Omega} |u|^{r(x)}\,dS \\
      & = \left( \frac{1}{p^+} - \frac{1}{r^-} \right) \int_{\partial\Omega} |u|^{r(x)}\,dS \\
      & \ge 0.
\end{split}
\end{equation*}
Therefore, $\F(u_k)\ge \tilde \F(v_k) + o(1).$ Let $\phi\in C^1(\bar\Omega)$. As $\tilde \F'(v_k)\to 0$, we have
\begin{align*}
o(1) &= \langle \tilde \F'(v_k), v_k\phi\rangle \\
&= \int_\Omega |\nabla v_k|^{p(x)}\phi\, dx - \int_{\partial\Omega} |v_k|^{r(x)}\phi\, dS + \int_\Omega |\nabla v_k|^{p(x)-2}\nabla v_k \nabla\phi v_k\, dx\\
&= A - B + C.
\end{align*}
Since $\{v_k\}_{k\in\N}$ is bounded in $W^{1,p(x)}(\Omega)$ and converges to $0$ in $L^{p(x)}(\Omega)$, it is easy to see, using Hölder inequality as stated in proposition \ref{Holder}, that $C\to 0$ as $k\to\infty$. 
Moreover by means of Lemma \ref{Brezis-Lieb}, (\ref{CCPtrace}), and (\ref{CCPtrace2}), there holds 
$$ A\to \int_\Omega\phi\, d\tilde\mu \qquad \mbox{and}\qquad B\to \int_{\partial\Omega} \phi\, d\tilde\nu, $$
where $\tilde\mu = \mu - |\nabla u|^{p(x)}\, dx$ and $\tilde\nu =\nu - |u|^{r(x)}\, dS$. So we conclude that $\tilde\mu=\tilde\nu$. In particular $\nu_i\ge \mu_i$ ($i\in I$) from where we obtain with (\ref{CCPtrace3}) that $\nu_i\ge \bar T_{x_i}^{-\frac{(N-1)p(x_i)}{p(x_i)-1}}$. Hence
\begin{equation*}
\begin{split}
 c &=  \lim_{k\to\infty} \F(u_k) \ge \lim_{k\to\infty} \tilde \F(v_k)  =\int \frac{1}{p(x)}\,d\tilde\mu - \int \frac{1}{r(x)}\,d\tilde\nu\\
 & = \int \Big(\frac{1}{p(x)}-\frac{1}{r(x)}\Big) \,d\tilde\nu = \sum_{i\in I} \left(\frac{1}{p(x_i)}-\frac{1}{p_*(x_i)}\right) \nu_i \\
 & \ge \#(I) \inf_{i\in I} \frac{p(x_i)-1}{(N-1)p(x_i)} \bar T_{x_i}^{-\frac{(N-1)p(x_i)}{p(x_i)-1}}.
\end{split}
\end{equation*}
We deduce that if $c< \inf_{i\in I} \frac{p(x_i)-1}{(N-1)p(x_i)} \bar T_{x_i}^{-\frac{(N-1)p(x_i)}{p(x_i)-1}}$ then $I$ must be empty implying that $u_k\to u$ strongly in $W^{1,p(x)}(\Omega)$.
\end{proof}

As a corollary, we can apply the Mountain--Pass Theorem to obtain the following necessary existence condition:

\begin{teo}\label{teoMPt}
Assume that $r^->p^+$ and that $h$ is such that $\J$ is coercive (see (\ref{coercive}) above). 
If there exists $v\in W^{1,p(x)}(\Omega)$ such that
\begin{equation}\label{CCPCondt}
 \sup_{s>0} \F(sv) < \inf_{x\in\A_T} \left(\frac{1}{p(x)} - \frac{1}{p_*(x)}\right) \bar T_{x}^{\frac{p(x)p_*(x)}{p_*(x)-p(x)}}
\end{equation}
then \eqref{MainEq traza} has a non-trivial nonnegative solution.
\end{teo}

\begin{proof}
The proof is an immediate consequence of the Mountain--Pass Theorem, Theorem \ref{teoPScondt} and assumption \eqref{CCPCondt}.
In fact, it suffices to verify that $\F$ has the Mountain--Pass geometry and that $\F(su)<0$ for some $s>0$. Concerning the latter condition notice that for $s>1$,
\begin{equation*}
\begin{split}
 \F(su) & = \int_\Omega \frac{s^{p(x)}}{p(x)}\left(|\nabla u|^{p(x)}+h(x)|u|^{p(x)}\right)\,dx
           - \int_{\partial\Omega} \frac{s^{r(x)}}{r(x)}|u|^{r(x)}\,dS \\
       & \le s^{p^+} \int_\Omega \frac{1}{p(x)}\left(|\nabla u|^{p(x)}+h(x)|u|^{p(x)}\right)\,dx
          - s^{r^-}\int_{\partial\Omega} \frac{1}{r(x)}|u|^{r(x)}\,dS,
\end{split}
\end{equation*}
which tends to $-\infty$ as $s\to +\infty$ since $r^->p^+$.

It remains to see that $\F$ has the Mountain--Pass geometry. Clearly $\F(0)=0$ and, 
if $\| v\|_{1,p(x)}=s$ is small enough, then 
$$ \int_\Omega |\nabla v|^{p(x)} + h |v|^{p(x)}\, dx \ge c_1 \| v\|_{1,p(x)}^{p^+} = c_1 s^{p^+} $$
since $\J$ is coercive, and on the other hand
$$ \|v\|_{r(x),\partial\Omega} \le C\| v\|_{1,p(x)} = Cs<1 $$
for $s$ small, so that 
$$ \int_{\partial\Omega} |v|^{r(x)}\, dS \le c_2\|v\|_{1,p(x)}^{r^-} = c_2 s^{r^-}. $$
Therefore
$$ \F(v) \ge \frac{c_1}{p^+} s^{p^+} - \frac{c_2}{r^-} s^{r^-} >0 $$
since $p^+<r^-$. This completes the proof.
\end{proof}

\section{Local conditions for \eqref{CCPCondt}}

In this section we provide  local conditions for \eqref{CCPCondt} to hold. These conditions are analogous to the ones found in \cite{FBSS2} where the critical problem for the $p(x)-$Laplacian with Dirichlet boundary condition was studied.

The idea is to evaluate $\F(sz_\ve)$ for  a suitable test function $z_\ve$ constructed by a scaled and truncated version of the extremal for $\bar K(N, p(x))^{-1}$ for a critical point $x\in \A_T$. Then, a refined asymptotic analysis will yield the desired result.

In order to construct the test function we need to recall the {\em Fermi coordinates} from differential geometry. Briefly speaking, the Fermi coordinates describe a neighborhood of a point $x_0\in\p\Omega$ with variables $(y,t)$ where $y\in\R^{N-1}$ are the coordinates in a local chart of $\p\Omega$ such that $y=0$ corresponds to $x_0$, and $t>0$ is the distance to $\partial\Omega$ along the unit inward normal vector.

\begin{defi}[Fermi Coordinates]\label{Fermi}\index{Fermi Coordinates}
We consider the following change of variables around a point $x_0\in\partial\Omega$.

We assume that $x_0=0$ and that $\partial\Omega$ has the following representation in a neighborhood $V$ of $0$:
$$
\partial\Omega\cap V = \{ x\in V\colon x_n = \psi(x'),\ x'\in U\subset \R^{N-1}\}, \ \ \Omega\cap V = \{ x\in V\colon x_n > \psi(x'),\ x'\in U\subset \R^{N-1}\}.
$$
The function $\psi\colon U\subset \R^{N-1}\to \R$ is assumed to be at least of class $C^2$ and that $\psi(0)=0$, $\nabla \psi(0)=0$.

The change of variables is then defined as $\Phi\colon U\times (0,\delta)\to \Omega\cap V$
$$ \Phi(y,t) = (y,\psi(y)) + t\nu(y), $$
where $\nu(y)$ is the unit inward normal vector, i.e.
$$ \nu(y) = \frac{(-\nabla \psi(y), 1)}{\sqrt{1+|\nabla\psi(y)|^2}}. $$
\end{defi}

\noindent It is well known that for $\delta>0$ small $\Phi$ defines a smooth diffeomorphism (see \cite{Escobar}).
For a general construction of the Fermi coordinates in differential manifolds, we refer to the book \cite{tubes}.

\medskip

Now, we are in position to construct the test functions needed in order to satisfy \eqref{CCPCondt}. Assume that $0\in \mathcal{A}_T\subset \p\Omega$. 
Then, the test-functions we consider are defined in the Fermi coordinates by 
$$ v_\eps(x)=\eta(y,t)V_{\eps, 0}(y,t),\quad x=\Phi(y,t), $$
where $V_{\eps,0}$ is defined in (\ref{Extremal}) by rescaling an extremal $V$ of $\bar K(N, p(0))^{-1}$, 
and $\eta\in C^\infty_c(B_{2\delta}\times [0,2\delta),[0,1])$ is a smooth cut-off function.
We normalize $v_\eps$ by considering the function $z_\eps$ defined by 
$$z_\ve=C v_\ve, \quad C= \bar K(N,p(0))^{-\frac{p(0)}{p(0)_*-p(0)}} \|V\|_{p(0)_*, \p\R^N_+}^{-1}.$$ 
With this choice of $C$, the function $Z(y,t):=C V(y,t)$ satisfies 
$$ \int_{\R^N_+} |\nabla Z|^{p(0)}\, dydt = \int_{\p\R^N_+} |Z|^{p(0)_*}\, dy 
                                     = \bar K(N,p(0))^{-\frac{p(0) p(0)_*}{p(0)_*-p(0)}}.
$$

\medskip 

From now on, we assume that $p\in \PP(\Omega)$ and $r\in \PP(\partial\Omega)$ are of class $C^2$, $0\in \p\Omega$ and we let $p:=p(0)$ and $r:=r(0)$.

In the propositions \ref{propIntLp}, \ref{propA} and \ref{propD.A} in the Appendix we compute some asymptotic expansions needed in order to properly evaluate $\F(sz_\ve)$. These propositions are fundamental in the proof of our next result. We choose to postpone their proofs to the appendix because they are technical and long.

\medskip

Eventually the following result provides a sufficient local condition for \eqref{CCPCondt} to hold:

\begin{teo}\label{teo.cond.local}
Assume that $r^->p^+$, and that $h$ is such that $\J$ is coercive. 
Assume moreover that there exists a point $x_0\in \A_T$ such that $\bar T = \bar T_{x_0}$ and such that $x_0$ is a local minimum of $p(x)$ and a local maximum of $r(x)$ and $p(x_0)<\min\{\sqrt{N},\frac{N^2}{3N-2}\}$. Assume eventually that one of the following conditions hold
\begin{enumerate}
 \item $\frac{\p p}{\p t}(x_0)>0$,
 \item $\frac{\p p}{\p t}(x_0)=0$ and $H(x_0)>0$ or
 \item $\frac{\p p}{\p t}(x_0)=0$, $H(x_0)=0$, $1<p(x_0)<2$ and $h(x_0)<0$ or
 \item $\frac{\p p}{\p t}(x_0)=0$, $H(x_0)=0$, $p(x_0)\geq 2$ and $\Delta p(x_0)>0$ or $\Delta_y r(x_0)<0$.
\end{enumerate}
Then there exists a nontrivial solution to \eqref{MainEq traza}. Here $\frac{\p p}{\p t}(x_0)=-\partial_\nu p(x_0)$ (with $\nu$ the unit exterior normal vector), $\Delta_y r(x_0):=\Delta (r\circ\Phi(\cdot,0))(0)$, and $\Delta p(x_0):=\Delta (p\circ\Phi)(0)$. 

\end{teo}

\noindent Notice that, as a consequence of the definition of the Fermi coordinates, we have that $\Delta_y r(x_0)$ coincides with the Laplacian of $r$ at $x_0$ for the natural metric of $\p\Omega$. 

\begin{proof}
We assume, without loss of generality that $x_0=0$ and denote $p=p(0)$. Observe that $r(0) = p_*$.

We first consider the case where $\partial_tp(0)>0$. In fact, from Propositions \ref{propIntLp}, \ref{propA} and \ref{propD.A}, we have
\begin{align*}
f_\varepsilon(s)=\F(s z_\varepsilon)&=\bar D_0+\bar D_1\varepsilon\ln\varepsilon- \bar A_0 + o(\varepsilon\ln\varepsilon)\\
&=f_0(s)+\varepsilon\ln\varepsilon f_1(s)+O(\varepsilon)
 \end{align*}
$C^1-$ uniformily in $s\in[0,s_0]$, with
$$
f_0(s)= \bar K(N,p)^{-\frac{p_* p}{p_*-p}}\left(\frac{s^p}{p}-\frac{s^{p_*}}{p_*}\right)
$$
 and
 $$
 f_1(s)=-\frac{N}{p}\frac{s^p}{p}\p_t p(0)\int_{\R^n_+} t|\nabla Z|^p \,dydt
 $$
Notice that $f_0$ reaches its maximum in $[0,s_0]$ at $s=1$. Moreover, it is a nodegenerate maximum since $f_0''(1)=(p-p_*) \bar K(N,p)^{-\frac{p_* p}{p_*-p}}\neq 0$. It follows that $f_\varepsilon$ reaches a maximum at $s_\varepsilon=1+a\varepsilon\ln\varepsilon+O(\varepsilon)$ for $a=-\frac{f_1'(1)}{f_0''(1)}$. Hence
$$
\sup_{s>0} \F(s z_\varepsilon) = \F(s_\varepsilon z_\varepsilon) = \left(\frac{1}{p} - \frac{1}{p_*}\right) \bar K(N,p)^{-\frac{p_* p}{p_*-p}}+f_1(1)\varepsilon\ln\varepsilon+O(\varepsilon)
$$
If $\p_t p(0)>0$ then $f_1(1)<0$ and the result follows.

Assume now that $\partial_t p(0)=0$ and $H(0)>0$. Then we have
\begin{align*}
\F(s z_\eps)&=\bar D_0+\bar D_2\varepsilon+ o(\eps) - \bar A_0\\
&=f_0(s)+f_2(s)\varepsilon+o(\varepsilon)
\end{align*}
$C^1-$ uniformily in $[0,s_0]$, with
$$
f_2(s)=- H(0) \frac{s^p}{p} \int_{\R^N_+}t|\nabla Z|^p \,dydt + \frac{H(0)}{N-1} s^p \int_{\R^N_+}\frac{t|y|^2}{r^2}|\nabla Z|^p \,dydt
$$
As before $f_\varepsilon$ reaches its maximum at $s_\varepsilon=1+a\varepsilon+o(\varepsilon)$ with $a=\frac{f_2'(1)}{f_0''(1)}$. So,
$$
\sup_{s>0}\F(s z_\varepsilon)=\F(s_\varepsilon z_\varepsilon)=  \left(\frac{1}{p} - \frac{1}{p_*}\right) \bar K(N,p)^{-\frac{p_* p}{p_*-p}} + f_2(1)\varepsilon+o(\varepsilon)
$$
So, we need that $f_2(1)<0$, i.e.
$$
- H(0) \frac{1}{p} \int_{\R^N_+}t|\nabla Z|^p \,dydt + \frac{H(0)}{N-1}  \int_{\R^N_+}\frac{t|y|^2}{r^2}|\nabla Z|^p \,dydt<0.
$$
But,
\begin{equation*}
\begin{split}
& -\frac{1}{p}\int_{\R^N_+}t|\nabla Z|^p \,dydt +\frac{1}{N-1} \int_{\R^N_+}\frac{t|y|^2}{r^2}|\nabla Z|^p \,dydt \\
& \le  \left(-\frac{1}{p}+\frac{1}{N-1}\right)\int_{\R^N_+}t|\nabla Z|^p \,dydt \\
& <  0
\end{split}
\end{equation*}
if $p<N-1$. So, since $H(0)>0$, the result follows.

Now suppose that $\partial_t p(0)=0$ and $H(0)=0$. Then
$$
\F(sz_\eps)=\bar D_0+\bar D_4\varepsilon^2\ln\varepsilon+o(\varepsilon^2\ln\varepsilon)+\bar C_0\varepsilon^p+o(\varepsilon^p)-\bar A_0 - \bar A_1\varepsilon^2\ln\varepsilon.
$$
If $1<p<2$
$$
\F(sz_\eps)=(\bar D_0-\bar A_0)+\bar C_0\varepsilon^p+o(\varepsilon^p) = f_0(s)+f_3(s)\varepsilon^p+o(\varepsilon^p)
$$
with
$$
f_3(s)=h(0)\frac{s^p}{p}\int_{\R^N_+} |\nabla Z|^p \,dydt.
$$
As before $f_\varepsilon$ reaches its maximum at $s_\varepsilon=1+a\varepsilon^p+o(\varepsilon^p)$ with $a=\frac{f_3'(1)}{f_0''(1)}$. Then,
$$
\sup_{s>0}\F(s z_\varepsilon) = \F(s_\varepsilon z_\varepsilon) =  \left(\frac{1}{p} - \frac{1}{p_*}\right) \bar K(N,p)^{-\frac{p_* p}{p_*-p}} + f_3(1) \varepsilon^p + o(\varepsilon^p)
$$
So, we need that $f_3(1)<0$. But, this is equivalent to $h(0)<0$.

If $p\geq 2$, we have
$$
\F(sz_\eps)=(\bar D_0 - \bar A_0) + (\bar D_4 - \bar A_1)\varepsilon^2\ln\varepsilon + o(\varepsilon^2\ln\varepsilon) = f_0(s) + f_4(s) \varepsilon^2\ln\varepsilon + o(\varepsilon^2\ln\varepsilon),
$$
with
\begin{align*}
f_4(s) = &-\frac{s^p}{p} \frac{N}{2p}\left(\partial_{tt} p(0) \int_{\R^N_+} t^2 |\nabla Z|^p\, dydt + \Delta_y p(0) \int_{\R^N_+} |y|^2 |\nabla Z|^p\, dydz\right)\\
& + \frac{s^{p_*}}{p_*} \frac{1}{2p_*} \Delta_y r(0) \int_{\p\R^N_+} |y|^2 Z^{p_*}\, dy.
\end{align*}

As before, we need that $f_4(1)<0$. Since $0$ is a local minimum of $p(x)$ and a local maximum of $r(x)$ and $\p_t p(0)=0$ it easily follows that $f_4(1)\le 0$. Moreover if one of the following inequalities
$$
\Delta_y r(0) \le 0\le \Delta p(0)
$$
is strict, then $f_4(1)<0$ and the result follows.
\end{proof}

\appendix

\section{Asymptotic expansions}\label{A.traza}

In this section we provide the asymptotic expansions needed in the proof of Theorem \ref{teo.cond.local}.

First we need the following asymptotic expansions for the Jacobian of the Fermi coordinates that are proved in \cite{Escobar}.
\begin{lema}\label{fermi.asymptotic} With the notation introduced in Definition \ref{Fermi}, the following asymptotic expansions hold
$$
J\Phi(y,t) = 1- Ht + O(t^2 + |y|^2),
$$
where $H$ is the mean curvature of $\partial\Omega$.

Also, if we denote $v(y,t) = u(\Phi(y,t))$,
$$
|\nabla u(x)|^2 = (\p_t v)^2+\sum_{i,j=1}^N\left(\delta^{ij}+2h^{ij}t+O(t^2+|y|^2)\right)\p_{y_i}v \p_{y_j}v,
$$
where $h^{ij}$ is the second fundamental form of $\partial\Omega$.
\end{lema}

The goal of this section is to prove the following propositions.

\begin{prop}\label{propIntLp} There holds
\begin{equation}\label{IntLp}
 \int_{\Omega} f(x) |v_\eps|^{p(x)} \,dx = \bar C_0\eps^p+o(\eps^p) \quad \text{ with } \quad \bar C_0 = f(0)\int_{\R^N_+} V^p\,dx.
\end{equation}
\end{prop}

\begin{prop}\label{propA} If $p<\frac{N-1}{2}$,
\begin{equation}\label{IntBoundary}
 \int_{\p \Omega} f(x) |v_\eps|^{r(x)} \,dS_x = \bar A_0 + \bar A_1 \eps^2 \ln\eps + o(\eps^2 \ln \eps)
\end{equation}
with
$$
\bar A_0=f(0)\int_{\R^{N-1}}V(y,0)^{p_*}\,dy,
$$
and
\begin{equation*}
\begin{split}
\bar A_1 & = -\dfrac{N-p}{2p}f(0)\int_{\R^{N-1}} (D^2r(0)y,y)V(y,0)^{p_*}\,dy \\
     & = -\dfrac{1}{2p_*}f(0)\Delta r(0) \int_{\R^{N-1}} |y|^2 V(y,0)^{p_*}\,dy.
\end{split}
\end{equation*}

\end{prop}

\begin{prop} \label{propD.A} Assume that $p<N^2/(3N-2)$. Then
 \begin{equation*}
\begin{split}
 \int_{\Omega} f(x)|\nabla v_\eps(x)|^{p(x)}\,dx
   = \bar D_0 + \bar D_1\eps\ln\eps + \bar D_2\eps + \bar D_3(\eps\ln\eps)^2 + \bar D_4\eps^2\ln\eps + O(\eps^2),
\end{split}
\end{equation*}
with
\begin{equation*}
\begin{split}
\bar D_0  = f(0)\int_{\R^N_+} |\nabla V|^p \,dydt, \quad
\bar D_1  = -\frac{N}{p}f(0)\p_tp(0)\int_{\R^N_+} t|\nabla V|^p \,dydt,
\end{split}
\end{equation*}
and, assuming that $\p_tp(0)=0$,
\begin{equation*}
\begin{split}
\bar D_2 & = (\p_tf(0)-Hf(0))\int_{\R^N_+}t|\nabla V|^p \,dydt + p\bar h f(0) \int_{\R^N_+}\frac{t|y|^2}{r^2}|\nabla V|^p \,dydt,  \\
\bar D_3 & = 0 \\
\bar D_4 & = -\frac{N}{2p}f(0)\p_{tt}p(0)\int_{\R^N_+}t^2|\nabla V|^p \,dydt
         - \frac{N}{2(N-1)p}f(0)\Delta_yp(0)\int_{\R^N_+}|y|^2|\nabla V|^p \,dydt
\end{split}
\end{equation*}

\end{prop}

\begin{proof}[Proof of Proposition \ref{propIntLp}]
We write
 $$ \int_{\Omega} f(x) |v_\eps|^{p(x)} \,dx = \int_{\R^N_+} f(y,t)|v_\eps(y,t)|^{p(y,t)}(1+O(|y|^2+|t|))\,dydt.   $$
Now the result follows as in \cite{FBSS2} Proposition 5.1.
\end{proof}

\begin{proof}[Proof of Proposition \ref{propA}]
We have
$$
\int_{\p\Omega} f v_\eps^{r(x)}\,dS = \int_{\R^{N-1}} f(y,\psi(y)) v_\eps(y,\psi(y))^{r(y,\psi(y))}(1+O(|y|^2))\,dy.
$$

Now the proof follows as in \cite{FBSS2} Proposition 5.1.
\end{proof}

To treat the gradient term, we need the following result:

\begin{lema} \label{propB}Assume $p<N^2/(3N-2)$ and that $p=p(y,t)$ has a local minimum at $(y,t)=(0,0)$.
Given a bounded $g\in  C^2(\Omega)$ and real numbers $a^{ij}$, $1\le i,j\le N-1$, we have
\begin{equation}\label{Gradient1}
\begin{split}
 \sum_{i,j=1}^{N-1} a^{ij}\int_{\R^N_+}& g(y,t)\eta(y,t)|\nabla V_\eps|^{p(y,t)-2}\p_iV_\eps(y,t)\p_jV_\eps(y,t)\,dydt \\
 & = \bar B_0 + \bar B_1\eps\ln\eps + \bar B_2 \eps + \bar B_3 (\eps\ln\eps)^2 + \bar B_4 \eps^2\ln\eps + O(\eps^2)
\end{split}
\end{equation}
where $\p_i=\frac{\p}{\p y_i}$, and
\begin{equation*}
\begin{split}
\bar B_0  =& \bar a g(0) \int_{\R^N_+} |\nabla V(y,t)|^p\frac{|y|^2}{r^2} \,dydt,  \quad
\bar B_1  = -\frac{N}{p}g(0)\p_tp(0)\bar a \int_{\R^N_+} |\nabla V(y,t)|^p\frac{|y|^2 t}{r^2} \,dydt \\
\bar B_2 =& \bar a \int_{\R^N_+} |\nabla V(y,t)|^p \frac{t |y|^2}{r^2} \left\{g(0)\p_tp(0)\ln |\nabla V(y,t)| +  \p_t g(0)\right\}\,dydt \\
\bar B_3 =& \frac{N^2}{2p^2}g(0)\p_tp(0)^2 \bar a \int_{\R^N_+} |\nabla V(y,t)|^p\frac{|y|^2 t^2 }{r^2}  \,dydt\\
\bar B_4 =& -\frac{N}{p}\bar a \int_{\R^N_+} |\nabla V(y,t)|^p \frac{|y|^2 t^2}{r^2}\left( -\frac{g(0)}{2} \p_{tt}p(0) + \p_tp(0)\p_tg(0) + \p_tp(0)^2g(0) \ln |\nabla V(y,t)| \right) \,dydt \\
& +\sum_{i=1}^{N-1} \frac{Ng(0)}{2p}a^{ii}\p_{ii}p(0)  \int_{\R^N_+} |\nabla V(y,t)|^p r^{-2} \left(y_1^4 -3 y_1^2y_2^2\right) \,dydt \\
 &  + \sum_{i,k=1}^{N-1} \frac{Ng(0)}{2p} \left(a^{ii}\p_{kk}p(0)+ 2a^{ik}\p_{ik}p(0)\right)  \int_{\R^N_+} |\nabla V(y,t)|^p r^{-2} y_1^2 y_2^2 \,dydt
\end{split}
\end{equation*}
where $\bar a = \frac{1}{N-1}\sum_{i=1}^{N-1} a^{ii}$ and $r=r(y,t) = \sqrt{(1+t)^2 + |y|^2}$.
\end{lema}

\begin{proof}
Notice that
$$
|\nabla V_\eps(y,t)|= \frac{N-p}{p-1} \eps^{\frac{N-p}{p(p-1)}}((\eps+t)^2+|y|^2)^{-\frac{N-1}{2(p-1)}}.
$$
So, $|\nabla V_\eps(y,t)|<1$ if $|(y,t)|>C\eps^\frac{N-p}{p(N-1)}$ where $C=\left(\frac{N-p}{p-1}\right)^\frac{p-1}{N-1}$, and $\nabla =(\nabla_y,\p_t)$. Moreover, since $p^-_{2\delta}=p:=p(0,0)$,
\begin{equation*}
\begin{split}
 & \int_{B^+_{2\delta}\backslash B_{C\eps^\frac{N-p}{p(N-1)}}} |\nabla V_\eps|^{p(y,t)-2} |\nabla_y V_\eps|^2 \,dydt  \le \int_{B^+_{2\delta}\backslash B_{C\eps^\frac{N-p}{p(N-1)}}} |\nabla V_\eps|^{p(y,t)} \,dydt \\
 & \le \int_{B^+_{2\delta}\backslash B_{C\eps^\frac{N-p}{p(N-1)}}} |\nabla V_\eps|^p \,dydt
  \le C \eps^{\frac{N-p}{p-1}}\int_{\R^N_+\backslash B_{C\eps^\frac{N-p}{p(N-1)}}} \left\{ (\eps+t)^2+|y|^2\right\}^{-\frac{p(N-1)}{2(p-1)}} \,dydt \\
 & \le C \eps^{\frac{N-p}{p-1}}\int_{\R^N\backslash B_{C\eps^\frac{N-p}{p(N-1)}}} |(y,t)|^{-\frac{p(N-1)}{p-1}} \,dydt
   \le C\eps^{\frac{N-p}{p-1}} \int_{C\eps^\frac{N-p}{p(N-1)}}^{+\infty} \rho^{N-1-\frac{p(N-1)}{p-1}}\,d\rho
\end{split}
\end{equation*}
Then, we obtain
$$
\int_{B^+_{2\delta}\backslash B_{C\eps^\frac{N-p}{p(N-1)}}} |\nabla V_\eps|^{p(y,t)-2}|\nabla_y V_\eps|^2 \,dydt
  \le C \eps^\frac{N}{p_*}.
$$
Since $p\le \frac{N^2}{3N-2}$, we get that $\tfrac{N}{p_*}\ge 2$, hence
$$
\int_{B^+_{2\delta}\setminus B_{C\eps^\frac{N-p}{p(N-1)}}} |\nabla V_\eps|^{p(x,t)-2}|\nabla_y V_\eps|^2 \,dydt = O(\eps^2).
$$
Hence
\begin{equation*}
\begin{split}
 & a^{ij}\int_{\R^N_+} g(y,t)\eta(y,t)|\nabla V_\eps|^{p(y,t)-2} \p_i V_\eps(y,t) \p_j V_\eps(y,t)\,dydt \\
 & = a^{ij} \int_{B^+_{C\eps^\frac{N-p}{p(N-1)} }} g(y,t) |\nabla V_\eps|^{p(y,t)-2}\p_iV_\eps(y,t)\p_jV_\eps(y,t)\,dydt
     + O(\eps^2) \\
 & = a^{ij}\int_{B^+_{C\eps^{-\frac{N(p-1)}{p(N-1)}}}} g(\eps y,\eps t) \eps^{N(1-\frac{p(\eps y,\eps t)}{p})} |\nabla V|^{p(\eps y,\eps t)-2}\p_iV \p_jV\,dydt + O(\eps^2).
\end{split}
\end{equation*}
Letting
$$
\phi_{ij}=|\nabla V|^{p-2} \p_iV \p_jV = |\nabla V(y,t)|^p \frac{y_i y_j}{r^2},\quad \nabla=(\nabla_y,\p_t),
$$
we obtain
\begin{equation*}
\begin{split}
\sum_{i,j=1}^{N-1} a^{ij}\int_{\R^n_+}& g(y,t)\eta(y,t)|\nabla V_\eps|^{p(y,t)-2} \p_iV_\eps \p_jV_\eps\,dydt \\
 & = \bar B_0(\eps) + \bar B_1(\eps)\eps\ln\eps + \bar B_2(\eps) \eps + \bar B_3(\eps) (\eps\ln\eps)^2 + \bar B_4(\eps) \eps^2\ln\eps + \eps^2 R(\eps)
\end{split}
\end{equation*}
 with coefficients $\bar B_i(\eps)$, $i=0,\dots,4$, defined as
\begin{align*}
\bar B_0 & = \sum_{i,j=1}^{N-1}a^{ij}g(0)\int_{\R^N_+} \phi_{ij}(y,t) \,dydt  \\
\bar B_1 & = -\frac{N}{p}g(0)\p_tp(0)\sum_{i,j=1}^{N-1}a^{ij}\int_{\R^N_+} t\phi_{ij}(y,t)\,dydt \\
\bar B_2 & = \sum_{i,j=1}^{N-2} a^{ij} \int_{\R^N_+}  \phi_{ij}(y,t)\big(g(0)t\p_tp(0)\ln |\nabla V| + \nabla g(0)(y,t)\big)\,dydt \\
\bar B_3 & = \frac{N^2}{2p^2}g(0)\p_tp(0)^2 \sum_{i,j=1}^{N-1}a^{ij}\int_{\R^N_+} t^2\phi_{ij}(y,t)\,dydt \\
\bar B_4 & = -\frac{N}{p}\sum_{i,j=1}^{N-1}a^{ij} \int_{\R^N_+} \phi_{ij}(x,t) \big(\frac{g(0)}{2}(D^2p(0)(y,t),(y,t)) + \p_tp(0)t (\nabla g(0),(y,t)) \\
       &\hskip5cm + \p_tp(0)^2g(0) t^2 \ln |\nabla V|\big) \,dydt,
\end{align*}
but with integral over $B^+_{C\eps^{\frac{N-p}{p(N-1)}-1}}$ instead of $\R^N_+$, and the error term $R(\eps)$ satisfies
\begin{equation*}
\begin{split}
 |R(\eps)| & \le C \int_{B^+_{C\eps^{-\frac{N(p-1)}{p(N-1)}}}} r^2 |\nabla V|^p\ln |\nabla V| (1+r \eps\ln\eps)\,dydt \\
         & \le C\int_{B^+_{C\eps^{-\frac{N(p-1)}{p(N-1)}}}}r^2 |\nabla V|^p\ln |\nabla V| \,dydt.
\end{split}
\end{equation*}
Clearly, this last integral is bounded by
$$
C \int_1^{+\infty} \rho^{1-\frac{N-p}{p-1}} \ln\rho\,d\rho
$$
which is finite since $p<\frac{N+2}{3}$. Moreover
\begin{equation*}
\begin{split}
 |\bar B_0-\bar B_0(\eps)| & \le C \int_{\R^N_+\backslash B^+_{\eps^{-\frac{N(p-1)}{p(N-1)}}}} |\nabla V|^p\,dydt
                   \le C \int_{\eps^{-\frac{N(p-1)}{p(N-1)}}}^\infty r^{-1-\frac{N-p}{p-1}} \,dr
                   \le C \eps^\frac{N(N-p)}{p(N-1)}
                   \le C \eps^2
\end{split}
\end{equation*}
since $p\le \frac{N^2}{3N-2}$. Also for $i=1,2$,
\begin{equation*}
\begin{split}
 |\bar B_i- \bar B_i(\eps)| & \le C \int_{\R^N_+\backslash B^+_{\eps^{-\frac{N(p-1)}{p(N-1)}}}} |(y,t)| (1+\ln\,|\nabla V|)|\nabla V|^p\,dydt \\
                  & \le C \int_{\eps^\frac{N(1-p)}{p(N-1)}}^\infty r^{1-\frac{N-p}{p-1}} \ln r \, dr   \\
                  & \le C \int_{\eps^\frac{N(1-p)}{p(N-1)}}^\infty r^{1-\frac{N-p}{p-1}+\alpha} \,dr  \quad\text{ for any }\alpha>0  \\
                  & \le C \eps^{\frac{N(N-2p+1)}{p(N-1)}-\beta}  \quad\text{ for any $\beta>0$ and if $p<\frac{N^2+N}{3N-1}$}, \\
                  & = o(\eps).
\end{split}
\end{equation*}
Eventually, for any $i=3,4$,
\begin{equation*}
\begin{split}
 |\bar B_i- \bar B_i(\eps)| & \le C \int_{\R^N_+\backslash B^+_{\eps^{-\frac{N(p-1)}{p(N-1)}}}} |(y,t)|^2 (1+\ln |\nabla V|)|\nabla V|^p\,dydt \\
                  & \le C \int_{\eps^{-\frac{N(p-1)}{p(N-1)}}}^\infty r^{1-\frac{N-p}{p-1}} \ln r\, dr   \\
                  & \le C \int_{\eps^{-\frac{N(p-1)}{p(N-1)}}}^\infty r^{1-\frac{N-p}{p-1}+\alpha} \,dr  \quad\text{ for any }\alpha>0  \\
                  &= o(1),
\end{split}
\end{equation*}
since $p<\frac{n+2}{3}$.

Hence if $p<N^2/(3N-2)$,
\begin{equation*}
\begin{split}
\sum_{i,j=1}^{N-1} a^{ij}\int_{\R^N_+} & g(y,t)\eta(y,t)|\nabla V_\eps|^{p(y,t)-2}\p_iV_\eps(y,t)\p_jU_\eps(y,t)\,dydt \\
 & = \bar B_0 + \bar B_1\eps\ln\eps + \bar B_2 \eps + \bar B_3( (\eps\ln\eps)^2 + \bar B_4 \eps^2\ln\eps + O(\eps^2).
\end{split}
\end{equation*}

Finally, using the radial symmetry in the $y$ variable, we can simplify the expressions for the $\bar B_i$'s.

For $\bar B_4$, notice that
\begin{equation*}
\begin{split}
\sum_{i,j=1}^{N-1} & a^{ij}\p_{kl}p(0) \int_{\R^N_+} |\nabla V|^p r^{-2} y_i y_j y^k y^l \,dydt \\
 =& \sum_{i=1}^{N-1} a^{ii}\p_{ii}p(0)  \int_{\R^N_+} |\nabla V|^p r^{-2} y_1^4 \,dydt
  + \left(\sum_{i\neq k}a^{ii}\p_{kk}p(0)+ 2a^{ik}\p_{ik}p(0)\right)  \int_{\R^N_+} |\nabla V|^pr^{-2} y_1^2 y_2^2 \,dydt \\
 =& \sum_{i=1}^{N-1} a^{ii}\p_{ii}p(0)  \int_{\R^N_+} |\nabla V|^pr^{-2} \left(y_1^4 -3 y_1^2 y_2^2\right) \,dydt \\
 &  + \sum_{i,k=1}^{N-1}\left(a^{ii}\p_{kk}p(0)+ 2a^{ik}\p_{ik}p(0)\right)  \int_{\R^N_+} |\nabla V|^pr^{-2} y_1^2 y_2^2 \,dydt
\end{split}
\end{equation*}

The other simplifications follow in the same manner.
\end{proof}

\begin{lema} \label{propC}Assume $p<N^2/(3N-2)$. There holds that
$$ \int_{\R^N_+} f(y,t)\eta(y,t)|\nabla V_\eps|^{p(y,t)}\,dydt
  = \bar C_0 + \bar C_1\eps\ln\eps + \bar C_2\eps + \bar C_3(\eps\ln\eps)^2 + \bar C_4\eps^2\ln\eps + O(\eps^2)
$$
with
\begin{equation*}
\begin{split}
\bar  C_0 =& f(0)\int_{\R^N_+} |\nabla V|^p \,dydt, \quad
\bar  C_1  = -\frac{N}{p}f(0)\p_tp(0) \int_{\R^N_+} t|\nabla V|^p\,dydt \\
\bar  C_2 =& \int_{\R^N_+} t|\nabla V|^p \left(f(0)\p_tp(0)\ln |\nabla V| + \p_tf(0)\right)\,dydt \\
\bar  C_3 =& \frac{N^2}{2p^2}f(0)\p_tp(0)^2 \int_{\R^N_+} t^2|\nabla V|^p\,dydt \\
\bar  C_4 =& -\frac{N}{p} \int_{\R^N_+} t^2|\nabla V|^p \left( \frac{f(0)}{2}\p_{tt}p(0)
       + \p_tp(0)\p_tf(0) + \p_tp(0)^2f(0)\ln |\nabla V| \right) \,dydt \\
      & - \frac{N}{2(N-1)p}f(0)\Delta_yp(0) \int_{\R^N_+} |y|^2|\nabla V|^p\,dydt,
      \quad \Delta_y=\sum_{i=1}^{n-1}\p_{ii}
\end{split}
\end{equation*}
\end{lema}

\begin{proof}
As before
$$
\int_{\R^N_+ \backslash B_{C\eps^\frac{N-p}{p(N-1)}}  } |\nabla V_\eps|^{p(y,t)}\,dydt
 \le  C \eps^\frac{N}{p_*} = O(\eps^2).
$$
so that
\begin{equation*}
\begin{split}
 & \int_{\R^N_+} f(y,t)\eta(y,t)|\nabla V_\eps|^{p(y,t)}\,dydt
  = \int_{B^+_{C\eps^\frac{N-p}{p(N-1)}}  } f(y,t)|\nabla V_\eps|^{p(y,t)}\,dydt + O(\eps^2)  \\
 & = \bar C_0(\eps) + \bar C_1(\eps)\eps\ln\eps + \bar C_2(\eps)\eps + \bar C_3(\eps)(\eps\ln\eps)^2 + \bar C_4(\eps)\eps^2\ln\eps + O(\eps^2)
\end{split}
\end{equation*}
where the constants $\bar C_i(\eps)$ are the same as
\begin{equation*}
\begin{split}
\bar C_0 & = f(0)\int_{\R^N_+} |\nabla V|^p \,dydt  \\
\bar C_1 & = -\frac{N}{p}f(0)\p_tp(0) \int_{\R^N_+} t|\nabla V|^p\,dydt \\
\bar C_2 & = \int_{\R^N_+} t|\nabla V|^p \left(f(0)\p_tp(0)\ln |\nabla V| + \p_tf(0)\right)\,dydt \\
\bar C_3 & = \frac{N^2}{2p^2}f(0)\p_tp(0)^2 \int_{\R^N_+} t^2|\nabla V|^p\,dydt \\
\bar C_4 & = -\frac{N}{p} \int_{\R^N_+} |\nabla V|^p \left( \frac{f(0)}{2}(D^2p(0)(y,t),(y,t))
       + \p_tp(0)\p_tf(0) t^2  + \p_tp(0)^2f(0) t^2 \ln |\nabla V| \right) \,dydt
\end{split}
\end{equation*}
but with integral over $B^+_{C\eps^{-\frac{N(p-1)}{p(N-1)}}}$ instead of $\R^N_+$. We can estimate $|\bar C_i(\eps)-\bar C_i|$ as we estimated $|\bar B_i(\eps)-\bar B_i|$ in the previous lemma.

Again, using the radial symmetry of $V$ we can simplify the constants $\bar C_i$ as in the previous lemma.
\end{proof}

With the aid of the previous Lemmas, we can now prove Proposition \ref{propD.A}.
\begin{proof}[Proof of Proposition \ref{propD.A}]
First, by Lemma \ref{fermi.asymptotic},
\begin{equation*}
\begin{split}
 \int_{\Omega} f(x)|\nabla v_\eps|^{p(x)}\,dx
  = \int_{\R^N_+} f(y,t)|\nabla v_\eps|^{p(y,t)} (1-Ht+O(t^2 + |y|^2))\,dydt,
\end{split}
\end{equation*}
where we denote $f(y,t) = f(\Phi(y,t))$ and $p(y,t)=p(\Phi(y,t))$.

Recall that, by Lemma \ref{fermi.asymptotic},
$$
|\nabla v_\eps|^2 = (\p_t v_\eps)^2+\sum_{i,j=1}^{N-1}\left(\delta^{ij}+2h^{ij}t+O(t^2+|y|^2)\right)\p_iv_\eps \p_jv_\eps, \quad \p_i=\frac{\p}{\p y_i}.
$$
Then
\begin{equation*}
\begin{split}
 & \int_{\R^N_+} f(y,t)|\nabla v_\eps|^{p(y,t)} (1-Ht+O(t^2+|y|^2))\,dydt \\
 & = \int_{\R^N_+} f(y,t)|\nabla (\eta V_\eps)|^{p(y,t)} (1-Ht+O(t^2+|y|^2))\,dydt  \\
 & = \int_{\R^N_+} f(y,t)\eta(y,t)^{p(y,t)} |\nabla V_\eps|^{p(y,t)} (1-Ht+O(t^2+|y|^2))\,dydt + R(\eps),
\end{split}
\end{equation*}
where
\begin{equation*}
\begin{split}
|R(\eps)|&\le C\int_{\R^N_+\setminus B_\delta} |V_\eps|^{p(y,t)}\, dydt\le C\eps^p \int_{\delta/\eps}^\infty r^{-\frac{p(N-p)}{p-1}+N-1}\, dr = O(\eps^2),
\end{split}
\end{equation*}
if $p\le (n+2)/3$. Hence
\begin{equation*}
\begin{split}
\int_{\Omega} f(x)|\nabla v_\eps|^{p(x)}\,dx  = \int_{\R^N_+} &f(y,t)\eta(y,t)^{p(y,t)}
   \Big[  (\p_tU_\eps)^2\\
    & + \sum_{i,j=1}^{N-1}(\delta^{ij}+2h^{ij}t+O(t^2+|y|^2))\p_iV_\eps \p_jV_\eps \Big]^\frac{p(y,t)}{2} \\
    &(1-Ht+O(t^2 + |y|^2))\,dydt + O(\eps^2)
\end{split}
\end{equation*}
with
\begin{equation*}
\begin{split}
 & \left[ (\p_t V_\eps)^2+\sum_{i,j=1}^{N-1}\left(\delta^{ij}+2h^{ij}t+O(t^2+|y|^2)\right)\p_iV_\eps \p_jV_\eps \right]^\frac{p(y,t)}{2} \\
 & = |\nabla V_\eps|^{p(y,t)} \left[ 1 + \sum_{i,j=1}^{N-1}p(y,t)t h^{ij} |\nabla V_\eps|^{-2} \p_iV_\eps \p_jV_\eps + O(t^2+|y|^2) \right] \\
 & = |\nabla V_\eps|^{p(y,t)} + p(y,t)t h^{ij}|\nabla V_\eps|^{p(y,t)-2} \p_iV_\eps \p_jV_\eps
    + |\nabla V_\eps|^{p(y,t)} O(t^2+|y|^2)
\end{split}
\end{equation*}
Then
\begin{equation*}
\begin{split}
 \int_{\Omega} f(x)|\nabla v_\eps|^{p(x)}\,dx
 =& \int_{\R^N_+} f(y,t)\eta(y,t)^{p(y,t)} |\nabla V_\eps|^{p(y,t)} \,dydt \\
 & +\sum_{i,j=1}^{N-1} h^{ij}\int_{\R^N_+} tf(y,t)p(y,t)\eta(y,t)^{p(y,t)} |\nabla V_\eps|^{p(y,t)-2} \p_iV_\eps \p_jV_\eps  \,dydt  \\
 & - H \int_{\R^N_+} tf(y,t)\eta(y,t)^{p(y,t)}  |\nabla V_\eps|^{p(y,t)} \,dydt \\
 & + O(\eps^2)
\end{split}
\end{equation*}
since
\begin{equation*}
\begin{split}
\int_{\R^N_+} |\nabla V_\eps|^{p(y,t)} O(t^2+|y|^2)\, dydt & \le C\int_{\R^N_+} |(y,t)|^2 |\nabla V_\eps|^{p(y,t)}  \,dydt\\
& \le C \eps^2 \int_{\R^N_+} |(y,t)|^2 |\nabla V|^{p+ O(\eps)}  \,dydt \\
& = C \eps^2 \int_{\R^N_+} |(y,t)|^2 |\nabla V|^p (1+ O(\eps)\ln |\nabla V|) \,dydt.
\end{split}
\end{equation*}
As before this last integral is finite provided that $p<(N+2)/3$.

The proof now follows applying Lemmas \ref{propB} and \ref{propC}.
\end{proof}

\section*{Acknowledgements}
This work was partially supported by Universidad de Buenos Aires under grant UBACYT 20020100100400 and by CONICET (Argentina) PIP 5478/1438.
A. Silva is a fellow of CONICET.

\bibliographystyle{plain}
\bibliography{biblio}

\def\ocirc#1{\ifmmode\setbox0=\hbox{$#1$}\dimen0=\ht0 \advance\dimen0
  by1pt\rlap{\hbox to\wd0{\hss\raise\dimen0
  \hbox{\hskip.2em$\scriptscriptstyle\circ$}\hss}}#1\else {\accent"17 #1}\fi}
  \def\ocirc#1{\ifmmode\setbox0=\hbox{$#1$}\dimen0=\ht0 \advance\dimen0
  by1pt\rlap{\hbox to\wd0{\hss\raise\dimen0
  \hbox{\hskip.2em$\scriptscriptstyle\circ$}\hss}}#1\else {\accent"17 #1}\fi}
\begin{thebibliography}{10}

\bibitem{Adi}
Adimurthi and S.~L. Yadava.
\newblock Positive solution for {N}eumann problem with critical nonlinearity on
  boundary.
\newblock {\em Comm. Partial Differential Equations}, 16(11):1733--1760, 1991.

\bibitem{Aubin}
Thierry Aubin.
\newblock Probl\`emes isop\'erim\'etriques et espaces de {S}obolev.
\newblock {\em C. R. Acad. Sci. Paris S\'er. A-B}, 280(5):Aii, A279--A281,
  1975.

\bibitem{Aubin2}
Thierry Aubin.
\newblock \'{E}quations diff\'erentielles non lin\'eaires et probl\`eme de
  {Y}amabe concernant la courbure scalaire.
\newblock {\em J. Math. Pures Appl. (9)}, 55(3):269--296, 1976.

\bibitem{Aubin3}
Thierry Aubin.
\newblock {\em Some nonlinear problems in {R}iemannian geometry}.
\newblock Springer Monographs in Mathematics. Springer-Verlag, Berlin, 1998.

\bibitem{Brezis-Lieb}
Ha{\"{\i}}m Br{\'e}zis and Elliott Lieb.
\newblock A relation between pointwise convergence of functions and convergence
  of functionals.
\newblock {\em Proc. Amer. Math. Soc.}, 88(3):486--490, 1983.

\bibitem{BN}
Ha{\"{\i}}m Br{\'e}zis and Louis Nirenberg.
\newblock Positive solutions of nonlinear elliptic equations involving critical
  {S}obolev exponents.
\newblock {\em Comm. Pure Appl. Math.}, 36(4):437--477, 1983.

\bibitem{Hebey}
Fran{\c{c}}oise Demengel and Emmanuel Hebey.
\newblock On some nonlinear equations involving the {$p$}-{L}aplacian with
  critical {S}obolev growth.
\newblock {\em Adv. Differential Equations}, 3(4):533--574, 1998.

\bibitem{libro}
Lars Diening, Petteri Harjulehto, Peter H{\"a}st{\"o}, and Michael
  R{\ocirc{u}}{\v{z}}i{\v{c}}ka.
\newblock {\em Lebesgue and {S}obolev spaces with variable exponents}, volume
  2017 of {\em Lecture Notes in Mathematics}.
\newblock Springer, Heidelberg, 2011.

\bibitem{DHL}
Zindine Djadli, Emmanuel Hebey, and Michel Ledoux.
\newblock Paneitz-type operators and applications.
\newblock {\em Duke Math. J.}, 104(1):129--169, 2000.

\bibitem{D}
Olivier Druet.
\newblock Generalized scalar curvature type equations on compact {R}iemannian
  manifolds.
\newblock {\em Proc. Roy. Soc. Edinburgh Sect. A}, 130(4):767--788, 2000.

\bibitem{DH}
Olivier Druet and Emmanuel Hebey.
\newblock The {$AB$} program in geometric analysis: sharp {S}obolev
  inequalities and related problems.
\newblock {\em Mem. Amer. Math. Soc.}, 160(761):viii+98, 2002.

\bibitem{Escobar}
Jos{\'e}~F. Escobar.
\newblock Conformal deformation of a {R}iemannian metric to a scalar flat
  metric with constant mean curvature on the boundary.
\newblock {\em Ann. of Math. (2)}, 136(1):1--50, 1992.

\bibitem{ER}
Pierpaolo Esposito and Fr{\'e}d{\'e}ric Robert.
\newblock Mountain pass critical points for {P}aneitz-{B}ranson operators.
\newblock {\em Calc. Var. Partial Differential Equations}, 15(4):493--517,
  2002.

\bibitem{Evans2}
Lawrence~C. Evans.
\newblock {\em Weak convergence methods for nonlinear partial differential
  equations}, volume~74 of {\em CBMS Regional Conference Series in
  Mathematics}.
\newblock Published for the Conference Board of the Mathematical Sciences,
  Washington, DC, 1990.

\bibitem{Faget}
Zo{\'e} Faget.
\newblock Best constants in {S}obolev inequalities on {R}iemannian manifolds in
  the presence of symmetries.
\newblock {\em Potential Anal.}, 17(2):105--124, 2002.

\bibitem{Fan}
Xianling Fan and Dun Zhao.
\newblock On the spaces {$L^{p(x)}(\Omega)$} and {$W^{m,p(x)}(\Omega)$}.
\newblock {\em J. Math. Anal. Appl.}, 263(2):424--446, 2001.

\bibitem{FBS}
Juli{\'a}n Fern{\'a}ndez~Bonder and Nicolas Saintier.
\newblock Estimates for the {S}obolev trace constant with critical exponent and
  applications.
\newblock {\em Ann. Mat. Pura Appl. (4)}, 187(4):683--704, 2008.

\bibitem{FBSS3}
Juli{\'a}n Fern{\'a}ndez~Bonder, Nicolas Saintier, and Anal{\'{\i}}a Silva.
\newblock On the sobolev trace theorem for variable exponent spaces in the
  critical range.
\newblock Submitted.

\bibitem{FBSS2}
Juli{\'a}n Fern{\'a}ndez~Bonder, Nicolas Saintier, and Analia Silva.
\newblock Existence of solution to a critical equation with variable exponent.
\newblock {\em Ann. Acad. Sci. Fenn. Math.}, 37:579--594, 2012.

\bibitem{FBSS1}
Juli{\'a}n Fern{\'a}ndez~Bonder, Nicolas Saintier, and Analia Silva.
\newblock On the {S}obolev embedding theorem for variable exponent spaces in
  the critical range.
\newblock {\em J. Differential Equations}, 253(5):1604--1620, 2012.

\bibitem{FBS1}
Juli{\'a}n Fern{\'a}ndez~Bonder and Anal{\'{\i}}a Silva.
\newblock Concentration-compactness principle for variable exponent spaces and
  applications.
\newblock {\em Electron. J. Differential Equations}, pages No. 141, 18, 2010.

\bibitem{Fu}
Yongqiang Fu.
\newblock The principle of concentration compactness in {$L^{p(x)}$} spaces and
  its application.
\newblock {\em Nonlinear Anal.}, 71(5-6):1876--1892, 2009.

\bibitem{tubes}
Alfred Gray.
\newblock {\em Tubes}.
\newblock Addison-Wesley Publishing Company Advanced Book Program, Redwood
  City, CA, 1990.

\bibitem{Harjuleto}
Petteri Harjulehto, Peter H{\"a}st{\"o}, Mika Koskenoja, and Susanna Varonen.
\newblock The {D}irichlet energy integral and variable exponent {S}obolev
  spaces with zero boundary values.
\newblock {\em Potential Anal.}, 25(3):205--222, 2006.

\bibitem{HV}
Emmanuel Hebey and Michel Vaugon.
\newblock Existence and multiplicity of nodal solutions for nonlinear elliptic
  equations with critical {S}obolev growth.
\newblock {\em J. Funct. Anal.}, 119(2):298--318, 1994.

\bibitem{KR}
Ondrej Kov{\'a}{\v{c}}ik and Ji{\v{r}}{\'{\i}} R{\'a}kosn{\'{\i}}k.
\newblock On spaces {$L^{p(x)}$} and {$W^{k,p(x)}$}.
\newblock {\em Czechoslovak Math. J.}, 41(116)(4):592--618, 1991.

\bibitem{MOSS}
Yoshihiro Mizuta, Takao Ohno, Tetsu Shimomura, and Naoki Shioji.
\newblock Compact embeddings for {S}obolev spaces of variable exponents and
  existence of solutions for nonlinear elliptic problems involving the
  {$p(x)$}-{L}aplacian and its critical exponent.
\newblock {\em Ann. Acad. Sci. Fenn. Math.}, 35(1):115--130, 2010.

\bibitem{Nazaret}
Bruno Nazaret.
\newblock Best constant in {S}obolev trace inequalities on the half-space.
\newblock {\em Nonlinear Anal.}, 65(10):1977--1985, 2006.

\bibitem{Saintier}
Nicolas Saintier.
\newblock Asymptotic estimates and blow-up theory for critical equations
  involving the {$p$}-{L}aplacian.
\newblock {\em Calc. Var. Partial Differential Equations}, 25(3):299--331,
  2006.

\bibitem{Saintier2}
Nicolas Saintier.
\newblock Estimates of the best {S}obolev constant of the embedding of
  {$BV(\Omega)$} into {$L^1(\partial\Omega)$} and related shape optimization
  problems.
\newblock {\em Nonlinear Anal.}, 69(8):2479--2491, 2008.

\bibitem{Saintier3}
Nicolas Saintier.
\newblock Best constant in critical {S}obolev inequalities of second-order in
  the presence of symmetries.
\newblock {\em Nonlinear Anal.}, 72(2):689--703, 2010.

\bibitem{Schoen}
Richard Schoen.
\newblock Conformal deformation of a {R}iemannian metric to constant scalar
  curvature.
\newblock {\em J. Differential Geom.}, 20(2):479--495, 1984.

\end{thebibliography}

\end{document}